\documentclass[11pt,reqno]{amsart}
\usepackage{color}

\usepackage{amsfonts,amscd}
\usepackage{amssymb}
\usepackage{url}
\usepackage[english]{babel}

\usepackage{tikz}
\theoremstyle{plain}
\newtheorem{theorem}                 {Theorem}      [section]

\newtheorem{lemma}        [theorem]  {Lemma}

\theoremstyle{definition}
\newtheorem{example}      [theorem]  {Example}
\newtheorem{remark}       [theorem]  {Remark}
\newtheorem{definition}   [theorem]  {Definition}

\numberwithin{equation}{section}

\def \R{{\mathbb R}}
\def \rn{{\mathbb R}}
\def \s{{\mathbb S}}

\def \C{{\mathbb C}}
\def \cn{{\mathbb C}}
\def \hn{{\mathbb H}}
\def \H{{\mathbb H}}
\def \nn{{\mathbb N}}
\def \N{{\mathbb N}}

\def \nab#1#2{\hbox{$\nabla$\kern -.3em\lower 1.0 ex
    \hbox{$#1$}\kern -.1 em {$#2$}}}

\def \SLR#1{\text{\bf SL}_{#1}(\rn)}
\def \SL2{\widetilde{\text{\bf SL}}_{2}(\rn)}

\def \SO#1{\text{\bf SO}(#1)}

\def \SU#1{\text{\bf SU}(#1)}

\def \Sp#1{\text{\bf Sp}(#1)}

\def \Sol{\text{\bf Sol}}
\def \Nil{\text{\bf Nil}}

\DeclareMathOperator{\Div}{div}

\numberwithin{equation}{section}

\begin{document}

\title[Proper $r$-harmonic functions on the Thurston geometries]
{Proper $r$-Harmonic functions on \\  the Thurston geometries}


\author{Sigmundur Gudmundsson}

\address{Mathematics, Faculty of Science\\ University of Lund\\
Box 118, Lund 221\\
Sweden}
\email{Sigmundur.Gudmundsson@math.lu.se}

\author{Anna Siffert}
\address{Max Planck Institute for Mathematics\\
Vivatsgasse 7\\
53111 Bonn\\
Germany}
\email{siffert@mpim-bonn.mpg.de}

\begin{abstract}
For any positive natural number $r\in\N^+$ we construct new explicit proper $r$-harmonic functions on the celebrated $3$-dimensional Thurston geometries $\Sol$, $\Nil$, $\SL2$, $\H^2\times\rn$ and $\s^2\times\rn$.
\end{abstract}



\maketitle

\section{Introduction}

For a positive natural number $r\in\N^+$, the complex-valued $r$-harmonic functions are solutions to a partial differential equation of order $2r$. This equation arises in various contexts, see for example the extensive analysis in the Lecture Notes of Mathematics volume \cite{Gaz-Gru-Swe}.  The best known applications are  in physics e.g. for $r=2$
in the areas of continuum mechanics, including elasticity theory and the solution of Stokes flows. 
The literature on $1$-harmonic and $2$-harmonic functions is vast, but usually the domains are either surfaces or open subsets of $\rn^n$ equipped with its standard flat Euclidean metric.

\smallskip

Recently, new explicit local $2$-harmonic functions were constructed on the classical compact simple Lie groups $\SU n$, $\SO n$ and $\Sp n$, see \cite{Gud-Mon-Rat-1} and \cite{Gud-Sif-1}.  This gives local solutions to the $2$-harmonic equation on the $3$-dimensional round sphere $\s^3\cong\SU 2$ and the standard hyperbolic space $\H^3$ via a general duality principle, see \cite{Gud-13} and \cite{Gud-Mon-Rat-1}.  This has been developed further in \cite {Gud-15} to construct proper $r$-harmonic functions on the hyperbolic spaces $\hn^n$ and spheres $S^n$ for any positive natural numbers $r,n\in\N^+$

The classical Riemannian manifolds $\rn^3$, $\s^3$ and $\H^3$ of constant curvature are all on Thurston's celebrated list of $3$-dimensional model geometries, see \cite{BW-book}, \cite{Sco} and \cite{Thu}.  In \cite{Gud-14} the investigation of $2$-harmonic functions to the other members on Thurston's list i.e. $\Sol$, $\Nil$, $\SL2$, $\H^2\times\rn$ and $\s^2\times\rn$ has been initiated.  

\smallskip

In this paper we construct new explicit $r$-harmonic functions
on the model geometries $\Sol$, $\Nil$, $\SL2$, $\H^2\times\rn$ and $\s^2\times\rn$ for any $r\in\N^+$. Our constructions are rather elementary. However, even for the special case when $r=2$, we get infinitely many new $2$-harmonic functions on the spaces $\Sol$ and $\Nil$.

\smallskip

Our main focus lies on the model geometries $\Sol$ and $\Nil$.
Both these geometries are described in terms of global coordinates $(x,y,t)$, see Section \ref{section-sol} and Section \ref{section-nil}, respectively.

\smallskip

For any natural number $r\geq 2$, we construct some multivariant polynomials on $\Sol$ which are $r$-harmonic functions.

\begin{theorem}\label{main-1}
For two natural numbers $m,n\in\N$ put $$r=\min(\lfloor m/2\rfloor,\lfloor n/2\rfloor)+2.$$
Then for $0\leq i\leq\lfloor m/2\rfloor$ and $0\leq j\leq\lfloor n/2\rfloor$, there exist complex numbers $c_{i,j}$ such that the function $f_{m,n}:\Sol\rightarrow\C$ given by 
\begin{equation*}
f_{m,n}(x,y,t)=\sum_{i=0}^{\lfloor m/2\rfloor}\sum_{j=0}^{\lfloor n/2\rfloor}c_{i,j}\cdot x^{m-2i}y^{n-2j}e^{2t(j-i)}
\end{equation*}
is proper $r$-harmonic.
\end{theorem}

In the case of the model space $\Nil$, our first main result is the following. 

\begin{theorem}\label{main-2}
Let $H_1:\R^2\rightarrow\R$ a proper harmonic map and $d,\alpha\in\N$ be given.
Then the function
$$f_{d,\alpha}(x,y,t)=H_1(x,y)\cdot x^d\cdot t^{\alpha}$$ 
is $(2\alpha+d+1)$-harmonic if $\partial_x^iH_1(x,y)\neq 0$ and $\partial_y^iH_1(x,y)\neq 0$ for all $i\in\N$.
\end{theorem}

Further, we prove that the product of any monomials in $x$, $y$ and $t$ is $r$-harmonic for some $r$ specified below.

\begin{theorem}\label{main-3}
Let $H_1:\R^2\rightarrow\R$ a proper harmonic map and $\alpha\in\N$ be given.
Then the function
$$f_{m,n,\alpha}(x,y,t)=x^m\cdot y^{n}\cdot t^{\alpha}$$ 
\begin{enumerate}
\item is $(\lfloor (m+\alpha)/2\rfloor+\lfloor n/2\rfloor+1+\lfloor \alpha/2\rfloor)$-harmonic if $\alpha$ is even,
\item is $(\lfloor (m+\alpha)/2\rfloor+\lfloor (n+1)/2\rfloor+1+\lfloor \alpha/2\rfloor)$-harmonic if $\alpha$ is odd.
\end{enumerate}
\end{theorem}

\smallskip

Throughout this paper, the results are formulated such that the solutions are globally defined, but clearly, the same constructions hold even locally.  This is particularly important for the holomorphic functions in use.

\vspace{0.5cm}

\noindent\textbf{Organisation.}
In Section\,\ref{section-r-harmonic} we provide some preliminaries. Section\,\ref{section-sol} and Section\,\ref{section-nil} are the main sections containing the proofs of Theorems\,\ref{main-1} and \ref{main-2}, respectively.  We supplement these results by Sections \ref{section-sl2} and \ref{section-product}, in which we provide an elementary construction of $r$-harmonic maps on $\SL2$  and the products $\H^2\times\rn$ and $\s^2\times\rn$, respectively. 

\vspace{0.5cm}

\noindent\textbf{Acknowledgments.}
The second author is grateful to the Max Planck Institute for Mathematics at Bonn for its hospitality and financial support.

\section{Proper $r$-harmonic functions}\label{section-r-harmonic}
Throughout this work, let $(M,g)$ be a smooth $m$-dimensional manifold equipped with a Riemannian metric $g$.  We complexify its tangent bundle $TM$ of $M$ to obtain $T^{\cn}M$ and extend the metric $g$ to a complex-bilinear form on $T^{\cn}M$.  Then the gradient $\nabla f$ of a complex-valued function $f:(M,g)\to\cn$ is a section of $T^{\cn}M$.  In this situation, the well-known {\it linear} Laplace-Beltrami operator $\tau$ on $(M,g)$ acts locally on $f$ as follows
$$
\tau(f)=\Div (\nabla f)=\sum_{i,j=1}^m\frac{1}{\sqrt{|g|}} \frac{\partial}{\partial x_j}
\left(g^{ij}\, \sqrt{|g|}\, \frac{\partial f}{\partial x_i}\right).
$$
For two complex-valued functions $f,h:(M,g)\to\cn$ we have the following well-known relation
\begin{equation*}
\tau(f\cdot h)=\tau(f)\cdot h+2\cdot\kappa(f,h)+f\cdot\tau(h),
\end{equation*}
where the {\it conformality} operator $\kappa$ is given by $\kappa(f,h)=g(\nabla f,\nabla h)$.  Locally this acts by
\begin{equation}\label{equation-products}
\kappa(f,h)=\sum_{i,j=1}^mg^{ij}\frac{\partial f}{\partial x_i}\frac{\partial h}{\partial x_j}.
\end{equation}
Below we will repeatedly use this formula, often without explicitly mentioning this.

\smallskip

With this preparation we can now recall the definition of the central objects that we are dealing with, namely the complex-valued $r$-harmonic functions.

\begin{definition}\label{definition-proper-r-harmonic}
For a natural number $r\in\nn$, the iterated Laplace-Beltrami operator $\tau^r$ is given by
$$\tau^{0} (f)=f\ \ \text{and}\ \ \tau^r (f)=\tau(\tau^{(r-1)}(f)).$$
A complex-valued function $f:(M,g)\to\cn$ is said to be
\begin{enumerate}
\item[(a)] {\it $r$-harmonic} if $\tau^r (f)=0$, and
\item[(b)] {\it proper $r$-harmonic} if $\tau^r (f)=0$ and $\tau^{(r-1)} (f)$ does not vanish identically.
\end{enumerate}
\end{definition}

Clearly, since any $r$-harmonic function is of course $n$-harmonic for all $n\geq r$, one is interested in the construction of proper $r$-harmonic maps.
Further, it should be noted that the {\it harmonic} functions are exactly $r$-harmonic for $r=1$
and the {\it biharmonic} functions are the $2$-harmonic ones.
In some texts, the $r$-harmonic functions are also called {\it polyharmonic} of order $r$.

\section{The model geometry $\Sol$}\label{section-sol}

The goal of this section is to provide explicit proper $r$-harmonic functions on the Thurston geometry $\Sol$ for any $r\in\N^+$. In the literature this problem has so far only been considered for $r=2$, see \cite{Gud-14}.
\vskip .2cm

$\Sol$ is a $3$-dimensional Riemannian homogeneous space modeled on the $3$-dimensional solvable Lie subgroup
$$
\Sol=\{
\begin{bmatrix}
e^t & 0 & x\\
0 & e^{-t} & y\\
0 & 0 & 1
\end{bmatrix}|\ x,y,t\in\rn\}
$$
of the special linear group $\SLR 3$.  The left-invariant Riemannian metric on $\Sol$ is determined by the orthonormal basis
$\{X,Y,T\}$ of its Lie algebra given by
$$
X=
\begin{bmatrix}
0 & 0 & 1\\
0 & 0 & 0\\
0 & 0 & 0
\end{bmatrix},\ \
Y=
\begin{bmatrix}
0 & 0 & 0\\
0 & 0 & 1\\
0 & 0 & 0
\end{bmatrix},\ \
T=
\begin{bmatrix}
1 &  0 & 0\\
0 & -1 & 0\\
0 &  0 & 0
\end{bmatrix}.
$$
In the global coordinates $(x,y,t)$ on $\Sol$ this takes the following well-known form
$$
ds^2=e^{2t}dx^2+e^{-2t}dy^2+dt^2.
$$
It is easily seen that the corresponding Laplace-Beltrami operator $\tau$ and the
conformality operator $\kappa$ satisfy
\begin{equation}\label{tension-Sol}
\tau(f)=e^{-2t}\frac {\partial ^2 f}{\partial x^2}
+e^{2t}\frac {\partial ^2 f}{\partial y^2}
+\frac {\partial ^2 f}{\partial t^2},
\end{equation}
and
$$
\kappa(f,h)=e^{-2t}\frac {\partial f}{\partial x}\frac {\partial h}{\partial x}
+e^{2t}\frac {\partial f}{\partial y}\frac {\partial h}{\partial y}
+\frac {\partial f}{\partial t}\frac {\partial h}{\partial t},
$$
respectively.
\vskip .2 cm

In his Examples 3.1 and 3.2 of \cite{Gud-14}, Gudmundsson constructs the first globally defined complex-valued proper $r$-harmonic functions on the model geometry $\Sol$.  They follow below.

\begin{example}\label{exam:Sol-1-1}
For non-zero elements $a,b\in\cn^{4}$ let the complex-valued functions
$f_1,f_2:\Sol\to\cn$ be defined by
$$f_1(x,y)=(a_1+a_2x+a_3y+a_4xy)$$
and
$$f_2(x,y)=(b_1+b_2x+b_3y+b_4xy).$$
Then, for any $r\in\N$, the function $F_r:\Sol\to\cn$ given by
$$F_r(x,y,t)=t^{2r}\cdot f_1(x,y)+t^{2r+1}\cdot f_2(x,y)$$
is proper $(r+1)$-harmonic.  The reader should note that there is a typo in Example 3.1 of \cite{Gud-14}, where $F_r$ was claimed to be proper $r$-harmonic.
\end{example}

\begin{example}\label{exam:Sol-1-2}
Let $H=h_2\cdot h_3:\Sol\to\cn$ be the product of the functions
$h_2,h_3:\Sol\to\cn$ with
$$h_{2}(x,y,t)=a_2(2x^2-e^{-2t})+a_3(2x^3-3xe^{-2t})$$
and
$$h_{3}(x,y,t)=b_2(2y^2-e^{2t})+b_3(2y^3-3ye^{2t}).$$
Then it is easily shown that $H$ is proper biharmonic.
Thus $H$ constitutes a complex 4-dimensional family of proper
biharmonic functions globally defined on the model space $\Sol$.
\end{example}

What follows should be seen as a generalisation of the results of \cite{Gud-14}. For each positive natural number $r\geq 2$ we construct infinite families of proper $r$-harmonic functions.  We achieve this goal by proving that for every $m,n\in\N^+$ there exist functions $f_{m,n}$ of the form
$$x^my^n+\mbox{\lq lower order terms\rq}$$
which are proper $(\min(\lfloor m/2\rfloor,\lfloor n/2\rfloor)+2)$-harmonic.
We start by giving a few examples.

\begin{example}
By straightforward computations one verifies that
the functions $f_{2,4},f_{2,5}:\Sol\to\cn$ defined by
$$f_{2,4}(x,y,t)=x^2y^4+\tfrac{3}{8}e^{4t}x^2-\tfrac{1}{2}e^{-2t}y^4+(\tfrac{21}{16}-3x^2y^2)e^{2t}$$
and
$$f_{2,5}(x,y,t)=x^2y^5+\tfrac{15}{8}e^{4t}x^2y-\tfrac{1}{2}e^{-2t}y^5+(\tfrac{105}{16}y-5x^2y^3)e^{2t}$$
are proper biharmonic.

\smallskip

Further, one verifies that the functions $f_{4,4},f_{5,4}:\Sol\to\cn$ defined by
\begin{equation*}
\begin{aligned}
f_{4,4}(x,y,t)=x^4y^4-3&e^{2t}x^4y^2-3e^{-2t}x^2y^4\\&+\frac{3}{8}e^{4t}x^4+\frac{63}{8}e^{2t}x^2+\frac{3}{8}e^{-4t}y^4+\frac{63}{8}e^{-2t}y^2
\end{aligned}
\end{equation*}
and
\begin{equation*}
\begin{aligned}
f_{5,4}(x,y,t)=
x^5y^4-&3e^{2t}x^5y^2-5e^{-2t}x^3y^4\\&+\frac{3}{8}e^{4t}x^5+
\frac{105}{8}e^{2t}x^3+\frac{15}{8}e^{-4t}xy^4+
\frac{315}{8}e^{-2t}xy^2
\end{aligned}
\end{equation*}
are proper $3$-harmonic.
\end{example}

We now turn to the proof of Theorem\,\ref{main-1}. Our strategy is to prove this result by induction.  To settle the initial step, with the following lemma, we provide Theorem \ref{main-1} for $m=0$ and $m=1$.

\begin{lemma}\label{m01}
Let $n\in\N$. Then for $0\leq k\leq\lfloor n/2\rfloor$ there exist 
\begin{enumerate}
\item
complex numbers $c_{0,k}$ such that the function $f_{0,n}:\Sol\rightarrow\C$ given by 
\begin{equation*}
f_{0,n}(x,y,t)=\sum_{k=0}^{\lfloor n/2\rfloor}c_{0,k}\cdot y^{n-2k}e^{2tk}
\end{equation*}
is proper harmonic;  
\item
complex numbers $d_{0,k}$ such that the function $f_{1,n}:\Sol\rightarrow\C$ given by 
\begin{equation*}
f_{1,n}(x,y,t)=\sum_{k=0}^{\lfloor n/2\rfloor}d_{0,k}\cdot xy^{n-2k}e^{2tk}
\end{equation*}
is proper harmonic.
\end{enumerate}
\end{lemma}
\begin{proof}
Below we use the convention $c_{0,k}=0$ if the index $k$ is negative or larger than $\lfloor n/2\rfloor$.  We start by proving part (1). By a simple calculation we have
\begin{equation*}
\begin{aligned}
\tau(f_{0,n})=\sum_{k=0}^{\lfloor n/2\rfloor+1}
((n+2-2k)(n+1-2k)c_{0,k-1}+4k^2c_{0,k})\cdot y^{n-2k}e^{2kt}.
\end{aligned}
\end{equation*}
Hence $f_{0,n}$ is harmonic if and only if
\begin{equation*}
\begin{aligned}
c_{0,k}=\tfrac{(-1)^k}{4^k(k!)^2}\tfrac{n!}{(n-2k)!}\cdot c_{0,0}
\end{aligned}
\end{equation*}
for $k\in\{1,\dots,\lfloor n/2\rfloor\}.$
Consequently, for this choice of the complex numbers $c_{0,k}$ and $c_{0,0}\neq 0$, the function $f_{0,n}$ is proper harmonic.

\smallskip

Since
$$\tau(f_{1,n})=x\cdot\tau(\sum_{k=0}^{\lfloor n/2\rfloor}d_{0,k}\cdot y^{n-2k}e^{2tk}),$$
and the argument of $\tau$ has the same form as $f_{0,n}$,
part (2) follows immediately from part (1). 
\end{proof}

\begin{remark}
Clearly, due to identity (\ref{tension-Sol}), a result analogous to Lemma \ref{m01} with the roles of $x$ and $y$ reversed  and $t$ substituted by $-t$ holds.  This means that for $m\in\N$ and $0\leq i\leq\lfloor m/2\rfloor$, there exist complex numbers $c_{i,0}$,  such that the function $f_{m,0}:\Sol\rightarrow\C$ given by 
\begin{equation*}
f_{m,0}(x,y,t)=\sum_{i=0}^{\lfloor m/2\rfloor}c_{i,0}\cdot x^{m-2i}e^{-2ti}
\end{equation*}
is proper harmonic.  Furthermore for $0\leq i\leq\lfloor m/2\rfloor$, there exist
complex numbers $d_{i,0}$, such that the function $f_{m,1}:\Sol\rightarrow\C$ given by 
\begin{equation*}
f_{m,1}(x,y,t)=\sum_{i=0}^{\lfloor m/2\rfloor}d_{i,0}\cdot x^{m-2i}ye^{-2ti}
\end{equation*}
is proper harmonic.
\end{remark}

With this preparation at hand we can now prove our first main result, Theorem\,\ref{main-1}, by induction.
\
\begin{proof}[Proof of Theorem \ref{main-1}]
We prove the claim by induction.  Below we assume without loss of generality that $m\leq n$. Note that we have already proven the claim for $(m,n)=(0,n)$ and $(m,n)=(1,n)$, see Lemma \ref{m01}. In order to carry out the induction step, we proceed in two steps as follows.
\begin{enumerate}
\item[(1.)]
 First we show that for $0\leq i\leq\lfloor m/2\rfloor$ and $0\leq k\leq {\lfloor n/2\rfloor}$ there exit coefficients $c_{i,0}$ and $c_{0,k}$ such that $\tau(f_{m,n})$ is of the form 
\begin{equation}\label{induction-step}
f_{m-2,n-2}=\sum_{i=0}^{\lfloor\frac{m-2}{2}\rfloor}\sum_{k=0}^{\lfloor\frac{n-2}{2}\rfloor}\hat{c}_{i,k}\cdot x^{m-2-2i}y^{n-2-2k}e^{2t(k-i)}
\end{equation}
for some $\hat{c}_{i,k}\in\C$. 
\item[(2.)]
In a second step we prove that for each choice of $\hat{c}_{i,k}$ there exist $c_{i,k}$ such that $\tau(f_{m,n})=f_{m-2,n-2}$.
This allows us to apply the induction assumption and thus to establish the claim.
\end{enumerate}

\smallskip

To accomplish (1.), consider the function $f_{m,n}$ with
\begin{equation}\label{koe-1}
c_{i,0}=\tfrac{(-1)^i}{\Pi_{j=1}^i(2j)^2}\tfrac{m!}{(m-2i)!}\cdot c_{0,0}\quad\mbox{for}\quad i\in\{1,\dots,{\lfloor m/2\rfloor}\}
\end{equation}
and
\begin{equation}\label{koe-2}
c_{0,k}=\tfrac{(-1)^k}{\Pi_{j=1}^k(2j)^2}\tfrac{n!}{(n-2k)!}\cdot c_{0,0}\quad\mbox{for}\quad k\in\{1,\dots,\lfloor n/2\rfloor\}.
\end{equation}
By a straightforward computation we yield
\begin{equation}\label{delta-fn}
\tau (f_{m,n})(x,y,t)=\sum_{i=0}^{\lfloor\frac{m}{2}\rfloor}\sum_{k=0}^{\lfloor\frac{n}{2}\rfloor}\tilde{c}_{i,k}\cdot x^{m-2i}y^{n-2k}e^{2t(k-i)}
\end{equation}
with
\begin{eqnarray*}
\tilde{c}_{i,k}&=&4(k-i)^2c_{i,k}+(n+2-2k)(n+1-2k)c_{i,k-1}\\
& &\qquad +(m+2-2i)(m+1-2i)c_{i-1,k}.
\end{eqnarray*}
Clearly, $\tilde{c}_{0,0}=0$. Further, by a simple computation using 
(\ref{koe-1}) and (\ref{koe-2}), we get $\tilde{c}_{0,k}=0$ for $k\in\{1,\dots,\lfloor n/2\rfloor\}$
and $\tilde{c}_{i,0}=0$ for $i\in\{1,\dots, \lfloor m/2\rfloor\}$.
In other words, $\tau (f_{m,n})(x,y,t)$ is of the desired form (\ref{induction-step})
where $\hat{c}_{i,k}=\tilde{c}_{i+1,k+1}$, i.e.
\begin{equation}\label{c-hat}
\begin{aligned}
\hat{c}_{i,k}=4(k-i)^2c_{i+1,k+1}+(n-2k)&(n-1-2k)c_{i+1,k}\\+&(m-2i)(m-1-2i)c_{i,k+1}.
\end{aligned}
\end{equation}
This finishes the first step of the proof.

\smallskip

Next we show (2.), i.e. we show that for any choice of $\hat{c}_{i,k}$ there exist $c_{i,k}$ such that $\tau (f_{m,n})=f_{m-2,n-2}.$ In other words, we want to prove that the $c_{i,k}$ can be expressed in terms of (linear combinations of) the $\hat{c}_{i,k}$.
Note that the indices of $\hat{c}_{i,k}$ are in the range $0\leq i\leq\lfloor m/2\rfloor-1$, $0\leq k\leq\lfloor n/2\rfloor-1$. Further, the range of the indices of $c_{i,k}$ are $0\leq i\leq\lfloor m/2\rfloor$, $0\leq k\leq\lfloor n/2\rfloor$, where those with either $i=0$ or $k=0$ are determined in terms of $c_{0,0}$ -- see equations (\ref{koe-1}) and (\ref{koe-2}).
Given $\hat{c}_{i_0,k_0}$ with $j_0=\max(i_0,k_0)$, we call all $\hat{c}_{i,k}$ with $i,j< j_0$ \textit{lower order terms (with respect to $\hat{c}_{i_0,k_0}$)}.

\smallskip

Although the proof is elementary, it is helpful to consider one example first.
 For this purpose let $m,n\in\N^+$ be such that $\lfloor m/2\rfloor=3$ and $\lfloor n/2\rfloor=5$.
In Figure\,1 we indicate the $c_{i,k}$, $0\leq i,k\leq 4$, by nodes (in different colors).
The equation (\ref{c-hat}) with $i=k=0$ determines $c_{0,0}$ in terms of $\hat{c}_{0,0}$.
Thus, by the identities (\ref{koe-1}) and (\ref{koe-2}), the nodes in yellow in Figure\,1 are all given in terms of $\hat{c}_{0,0}$.
Next, we consider the three equations (\ref{c-hat}) with $(i,k)=(1,0),(1,1),(0,1)$. They determine the nodes in red in terms of $\hat{c}_{0,0}$,  $\hat{c}_{1,0}$,  $\hat{c}_{0,1}$ and  $\hat{c}_{1,1}$.
We can now proceed analogously, i.e. the five equations (\ref{c-hat}) with $(i,k)=(2,0),(2,1),(2,2),(1,2),(0,2)$ determine nodes in blue in terms of $\hat{c}_{0,0}$, $\hat{c}_{1,0}$, $\hat{c}_{0,1}$, $\hat{c}_{1,1}$, $\hat{c}_{2,0}$, $\hat{c}_{2,1}$, $\hat{c}_{2,2}$, $\hat{c}_{1,2}$ and $\hat{c}_{0,2}$.
The equations (\ref{c-hat}) with $(i,k)=(0,3)$ and $(i,k)=(1,3)$ successively determine $c_{1,4}$ and $c_{2,4}$ (i.e. the two green nodes on the left) in terms of the same $\hat{c}_{i,k}$'s and $\hat{c}_{0,3}$, $\hat{c}_{1,3}$. Evaluating equation (\ref{c-hat}) for $(i,k)=(2,3)$ gives as a linear combination of $c_{3,3}$ and $c_{3,4}$ in terms of lower order terms. In other words, we have one degree of freedom here. The equations (\ref{c-hat}) for $(i,k)=(0,4), (1,4),(2,4)$ successively determine the nodes in violet. Clearly, the one degree of freedom stems from the fact that 
we have a constant term in $f_{m,n}$ which of course vanishes upon applying $\tau$.

\smallskip

\begin{figure}\label{fig1}
\begin{center}
  \begin{tikzpicture}
    \foreach \i in {0,...,5}
      \path[black] (-1,\i) node{\i};
      \foreach \j in {0,...,3}
      \path[black] (\j,-1) node{\j};
    \foreach \i in {0,...,3}
      \foreach \j in {0,...,5}{
        \draw (\i,\j) circle(3pt);
          \fill[red] (1,2) circle(2pt);
              \fill[yellow] (1,0) circle(2pt);
                  \fill[red] (2,1) circle(2pt);
                      \fill[yellow] (2,0) circle(2pt);
                          \fill[yellow] (0,2) circle(2pt);
                              \fill[yellow] (0,1) circle(2pt);
 \fill[yellow] (0,0) circle(2pt);
 \fill[red] (1,1) circle(2pt);
 \fill[blue] (1,3) circle(2pt);
\fill[blue] (2,3) circle(2pt);
\fill[blue] (3,2) circle(2pt);
\fill[blue] (3,1) circle(2pt);
 \fill[blue] (2,2) circle(2pt);
 \fill[green] (1,4) circle(2pt);
\fill[green] (2,4) circle(2pt);
\fill[green] (3,4) circle(2pt);
 \fill[green] (3,3) circle(2pt);
    \fill[violet] (1,5) circle(2pt);
 \fill[violet] (2,5) circle(2pt);
 \fill[violet] (3,5) circle(2pt);
 \ifnum \i = 0
          \ifnum \j > 0
 \fill[yellow] (0,\j) circle(2pt);
          \fi
        \fi
 \ifnum \i >0
          \ifnum \j = 0
 \fill[yellow] (\i,0) circle(2pt);
          \fi
        \fi
      };
  \end{tikzpicture}
\end{center}
\caption{$\lfloor m/2\rfloor=3$ and $\lfloor n/2\rfloor=5$.}
\end{figure}
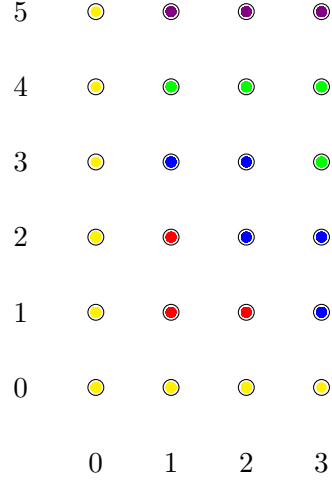

\smallskip

The preceding example indicates how to treat the general case which we consider from now on.
Let $s\in\N$ with $0\leq s\leq \lfloor m/2\rfloor-1$ be given.
We claim that the $2s+1$ equations (\ref{c-hat}) with $$(i,k)=(s,0),(s,1),\dots, (s,s), (s-1,s),\dots, (1,s),(0,s)$$
determine
$c_{i,k}$ with $$(i,k)=(s+1,1),\dots, (s+1,s), (s,s), (s,s+1),(s-1,s+1), \dots, (1,s+1)$$
in terms of the $\hat{c}_{i,k}$ with $i,k\leq s$.\\
We prove the claim by induction on $s$. In order to settle the initial induction step, consider equation (\ref{c-hat}) for $s=0$, i.e. $i=k=0$. This is equivalent to 
\begin{equation}\label{c00}
\hat{c}_{0,0}=n(n-1)c_{1,0}+m(m-1)c_{0,1}=-m(m-1)n(n-1)/2\cdot c_{0,0}.
\end{equation}
Thus $c_{0,0}$ is determined by $\hat{c}_{0,0}.$
Suppose now that we have shown the claim for some $0\leq s_0\leq \lfloor m/2\rfloor-1$ 
and further assume that $s_0+1\leq \lfloor m/2\rfloor-1$.
Evaluating (\ref{c-hat}) for $(i,k)=(0,s_0+1)$ gives $c_{1,s_0+2}$ in terms of $c_{0,s_0+2}$, $c_{1,s_0+1}$ and $\hat{c}_{0,s_0+1}$. By induction assumption $c_{1,s_0+1}$ can be expressed in terms of lower order $\hat{c}_{i,k}$'s. Due to (\ref{c00}) and (\ref{koe-2}) this also applies to $c_{0,s_0+2}$. Analogously, we can successively determine $c_{i,s_0+2}$ with $1\leq i\leq s_0$
in terms of lower order $\hat{c}_{i,k}$'s.
Similarly, we can determine $c_{s_0+2,k}$ with $1\leq k\leq s_0$
in terms of lower order $\hat{c}_{i,k}$'s.
It remains to show that upon 
evaluating (\ref{c-hat}) for $(i,k)=(s_0,s_0+1), (s_0+1,s_0+1)$ and $(s_0+1,s_0)$ yields equations
for $c_{s_0+1,s_0+2}$, $c_{s_0+1,s_0+1}$ and $c_{s_0+2,s_0+1}$.
Indeed, we have 
\begin{equation}
\begin{pmatrix} 
\hat{c}_{s_0,s_0+1}\\ 
\hat{c}_{s_0+1,s_0+1}\\
\hat{c}_{s_0+1,s_0}
\end{pmatrix}=A\cdot\begin{pmatrix} 
{c}_{s_0+1,s_0+1}\\ 
{c}_{s_0+1,s_0+2}\\
{c}_{s_0+2,s_0+1}
\end{pmatrix}+\mbox{lower order terms}
\end{equation}
with
\begin{equation*}
A=\begin{pmatrix} 
x&4&0\\ 
0&y&x\\
y&0&4
\end{pmatrix},
\end{equation*}
 $x=(n-2-2s_0)(n-3-2s_0)$ and $y=(m-2-2s_0)(m-3-2s_0)$.
The matrix $A$ is invertible since $\det A=8xy$ and $s_0\leq \lfloor m/2\rfloor-2\leq \lfloor n/2\rfloor-2$ by assumption.
\smallskip

The preceding considerations prove that we can express all $c_{i,k}$ with $i,k\leq  \lfloor m/2\rfloor$ in terms of lower order $\hat{c}_{i,k}$'s, where we have one degree of freedom.

\smallskip

We now evaluate the equation (\ref{c-hat}) successively for 
\begin{equation*}
\begin{aligned}
(i,k)=&(0,\lfloor m/2\rfloor), \dots, (\lfloor m/2\rfloor-1,\lfloor m/2\rfloor),\\& (0, \lfloor m/2\rfloor+1),\dots, (\lfloor m/2\rfloor-1,\lfloor m/2\rfloor+1),\\&\hspace{3cm}\dots\hspace{3cm},\\&(0,\lfloor n/2\rfloor-1),\dots, (\lfloor m/2\rfloor-1,\lfloor n/2\rfloor-1)
\end{aligned}
\end{equation*}
and thus determine the remaining $c_{i,k}$ in terms of $\hat{c}_{i,k}$'s.
This settles the second step which allows us to apply the induction assumption to (\ref{delta-fn}) and thus to establish the claim.
\end{proof}

\section{The model geometry $\Nil$}\label{section-nil}

The aim of this section is to provide explicit proper $r$-harmonic functions on the Thurston geometry $\Nil$ for any $r\in\N^+$. In the literature this problem has so far only been considered in \cite{Gud-14}, in the case of $r=2$.
\vskip .2cm

$\Nil$ is a $3$-dimensional Riemannian homogeneous space modelled on the classical nilpotent Heisenberg group.  It can be presented as the Lie subgroup
$$
\Nil=\{
\begin{bmatrix}
1 & x & t\\
0 & 1 & y\\
0 & 0 & 1
\end{bmatrix}|\ x,y,t\in\rn\}
$$
of the special linear group $\SLR 3$ equipped with its standard left-invariant Riemannian metric.
The restriction of this metric to $\Nil$ is completely determined by the orthonormal
basis $\{X,Y,T\}$ of its Lie algebra, which is given by
$$
X=
\begin{bmatrix}
0 & 1 & 0\\
0 & 0 & 0\\
0 & 0 & 0
\end{bmatrix},\ \
Y=
\begin{bmatrix}
0 & 0 & 0\\
0 & 0 & 1\\
0 & 0 & 0
\end{bmatrix},\ \
T=
\begin{bmatrix}
0 & 0 & 1\\
0 & 0 & 0\\
0 & 0 & 0
\end{bmatrix}.
$$

In the global coordinates $(x,y,t)$ on $\Nil$ the left-invariant Riemannian metric satisfies
$$
ds^2=dx^2+dy^2+(dt-xdy)^2.
$$

A straightforward calculation shows that the associated Laplace-Beltrami operator $\tau$ on $\Nil$ is given by
\begin{equation}\label{tension-Nil}
\tau(f)=(\frac {\partial ^2 f}{\partial x^2}+\frac {\partial ^2 f}{\partial y^2})
+2x\,\frac {\partial ^2 f}{\partial y\partial t}+(1+x^2)\,\frac {\partial ^2 f}{\partial t^2}.
\end{equation}
Furthermore, the conformality operator $\kappa$ satisfies
\begin{equation}\label{kappa-Nil}
\kappa(f,h)=\frac{\partial f}{\partial x}\frac{\partial h}{\partial x}+\frac{\partial f}{\partial y}\frac{\partial h}{\partial y}+x\,(\frac{\partial f}{\partial y}\frac{\partial h}{\partial t}+\frac{\partial f}{\partial t}\frac{\partial h}{\partial y})+(1+x^2)\,\frac{\partial f}{\partial t}\frac{\partial h}{\partial t}.
\end{equation}

\medskip

In his paper \cite{Gud-14}, Gudmundsson constructed the first examples of biharmonic functions on $\Nil$. He shows that for a non-zero element $b\in\cn^{12}$ the function $B:\Nil\to\cn$ given by
\begin{eqnarray*}
B(x,y,t)&=&b_1x^2+b_2y^2+b_3yt+b_4x^3+b_5x^2y+b_6x^2t+b_7xy^2\\
& &\quad+b_8y^3+b_9x^3y+b_{10}xy^3+b_{11}y^2t+b_{12}x^3t
\end{eqnarray*}
yields a complex 12-dimensional family of proper biharmonic functions on the model space $\Nil$.

\smallskip

Below we construct proper $r$-harmonic functions on $\Nil$ for any $r\in\N^+$. In the following example we first consider products of proper harmonic maps of $(x,y)$ in the Euclidean $\rn^2$ and monomials in $t\in\rn$.

\begin{example}
Let the function $H:\R^2\rightarrow\R$ be defined by 
$$H(x,y)=\frac 12\cdot (e^{(x+iy)}+e^{(x-iy)}).$$
Then $H$ is a non-vanishing sum of a holomorphic function and a anti-holomorphic one.  This means that it is proper harmonic and clearly satisfies 
$$H(x,y)=e^x\cdot\cos y.$$
For small natural numbers $d\in\N$, straightforward calculations verify that the function $f_d:\Nil\rightarrow\C$ given by $$f_d(x,y,t)=e^x\cdot \cos(y)\cdot t^d$$ is proper $(2d+1)$-harmonic.
\end{example}

In what follows we will provide Theorem \ref{main-2}.
For this purpose we first establish our preparatory Lemma \ref{prep} in which we determine the degree of the harmonicity for products of harmonic functions in $(x,y)$ in the Euclidean $\rn^2$ and monomials in $x$.
Such a result is most likely known, however, we provide here a proof for the reader's convenience.

\begin{lemma}\label{prep}
Let $H:\R^2\rightarrow\cn$ be a proper harmonic function from the Euclidean $\rn^2$ and $d\in\N$. Further assume that $\partial_x^kH(x,y)\neq0$ for all $k\in\N$.  Then the product $$f_{d}(x,y)=H(x,y)\cdot x^d$$
is proper $(d+1)$-harmonic.
\end{lemma}

\begin{proof}
We prove the claim by induction on the degree $d$. By assumption $f_0=H$ which is proper harmonic.  This settles the initial step. 

For the general induction step let us assume that $H:\R^2\rightarrow\R$ is a proper harmonic map with $$\partial_x^kH(x,y)\neq 0$$ for all $k\in\N$. Furthermore, that we have proven the claim for all $d\leq d_0$, for some $d_0\in\N^+$.
From equation (\ref{tension-Nil}) we have
\begin{eqnarray*}\label{der}
& &\tau(f_{d_0+1}(x,y))\\
&=&\tau(H(x,y))+2(\partial_xH)(x,y)(d_0+1)x^{d_0}+H(x,y)d_0(d_0+1)x^{d_0-1}\\
&=&2(d_0+1)(\partial_xH)(x,y)x^{d_0}+d_0(d_0+1)H(x,y)x^{d_0-1}.
\end{eqnarray*}
Since $\partial_x^kH(x,y)\neq0$ for all $k\in\N$, we obviously also have
 $\partial_x^k(\partial_xH)(x,y)\neq0$ for all $k\in\N$. Hence,
by the induction assumption, the first and the second term on the right hand side of (\ref{der}) are proper $(d_0+1)$-harmonic and proper $d_0$-harmonic, respectively.
Thus the function $f_{d_0+1}$ is proper $(d_0+2)$-harmonic, which establishes the result.
\end{proof}

With this preparation at hand we are now ready to prove Theorem \ref{main-2}.

\begin{proof}[Proof of Theorem \ref{main-2}]
We prove the claim by induction on the exponent $\alpha$. Lemma \ref{prep} provides the claim for $\alpha=0$ and thus settles
the initial induction step.

Let $\alpha_0\in\N^+$ be given. Below we will call the terms which factor
$t^{\alpha_0+1}$ of \textit{Type I}. All terms which factor powers of $t$ which are strictly smaller than $\alpha_0+1$ are said to be of \textit{Type II}. Assume now that the claim has been proven for all $\alpha\leq \alpha_0$.
From (\ref{tension-Nil}) we have
\begin{eqnarray*}\label{delf}
& &\tau(f_{d,\alpha_0+1}(x,y,t))\\
&=&\tau(H(x,y))\cdot x^d)\cdot t^{\alpha_0+1}+2x(\alpha_0+1)\partial_y H(x,y)x^dt^{\alpha_0}\\
& &\qquad +\alpha_0(\alpha_0+1)(1+x^2)H(x,y)\cdot x^d\cdot t^{\alpha_0-1}.
\end{eqnarray*}
Here the first term $\tau(H(x,y))\cdot x^d)\cdot t^{\alpha_0+1}$ is of \textit{Type I} and the rest of \textit{Type II} and hence we can ignore the latter by the induction assumption. This means that we only have to determine the degree of harmonicity of the first term $$R_1(x,y,t)=\tau(H(x,y)\cdot x^d)\cdot t^{\alpha_0+1}.$$
Since $R_1$ is of Type II we can not yet apply the induction assumption for this.  Instead we first have to apply $\tau$ again to $R$.
From (\ref{tension-Nil}) we have
\begin{eqnarray*}
& &\tau(R(x,y,t))\\
&=&\tau^2(H(x,y)\cdot x^d)\cdot t^{\alpha_0+1}
+2x(\alpha_0+1)\partial_y(\tau( H(x,y)\cdot x^d))t^{\alpha_0}\\
& &\quad +(\alpha_0+1)\alpha_0(1+x^2)\tau(H(x,y)\cdot x^d)\cdot t^{\alpha_0-1}.
\end{eqnarray*}
Again using again the induction assumption, one easily proves that the function 
$$2(\alpha_0+1)\partial_y(\tau( H(x,y)\cdot x^d))t^{\alpha_0}+(\alpha_0+1)\alpha_0(1+x^2)\tau(H(x,y)\cdot x^d)\cdot t^{\alpha_0-1}$$
is $(2\alpha_0+d)$-harmonic.
 Hence we only have to determine the degree of harmonicity of the remaining term $R_2(x,y,t)=\tau^2(H(x,y)\cdot x^d)\cdot t^{\alpha_0+1}.$
We now want to apply the same argumentation as above to show that we can
 reduce the problem to determining the degree of harmonicity of a remaining term of the form $R_3(x,y,t)=\tau^3(H_1(x,y)\cdot x^d)\cdot t^{\alpha_0+1}$. This reduction step is now proven in full generality. For this purpose suppose that we have a remaining term 
 $$R_i(x,y,t)=\tau^i(H_1(x,y)\cdot x^d)\cdot t^{\alpha_0+1}$$
of Type II, for some $i\in\N^+$.
Applying the Laplace-Beltrami operator $\tau$ to $R_i$ yields
\begin{eqnarray*}
& &\tau(R_i(x,y,t))\\
&=&\tau^{i+1}(H_1(x,y)\cdot x^d)\cdot t^{\alpha_0+1}+2(\alpha_0+1)x\partial_y(\tau^{i}(H_1(x,y)\cdot x^d))t^{\alpha_0}\\
& &\quad +\alpha_0(\alpha_0+1)(1+x^2)\tau^{i}(H_1(x,y)\cdot x^d)t^{\alpha_0-1}.
\end{eqnarray*}
In order to apply the induction assumption we will use that for any $n\in\N$ we have
\begin{equation}\label{deln}
\tau^n(H_1(x,y)\cdot x^d)=\sum_{j=0}^nc_j(d)(\partial_x^jH_1)(x,y)x^{d-2n+j},
\end{equation}
where $c_j(d)$, $j\in\{0,\dots,n\}$, is some positive constant depending on $d$. This can be proven easily by induction on $n$.

\smallskip

Consequently, by induction assumption and equation (\ref{deln}) we have that 
$$x\partial_y(\tau(H_1(x,y)\cdot x^d)^{i})t^{\alpha_0}$$
is $(d-i+2+2\alpha_0)$-harmonic. Since we applied $\tau$ $(i+1)$-times to $f_{d,\alpha_0+1}(x,y,t)$, we have that $f_{d,\alpha_0+1}(x,y,t)$ is at least $(2(\alpha_0+1)+d+1)$-harmonic.
Similarly, by induction assumption and equation (\ref{deln}) we have that 
$$(1+x^2)\tau(H_1(x,y)\cdot x^d)^{i}t^{\alpha_0-1}$$
is $(d-i+3+2(\alpha_0-1))$-harmonic. Since we applied $\tau$ $(i+1)$-times
to $f_{d,\alpha_0+1}(x,y,t)$, we have that $f_{d,\alpha_0+1}(x,y,t)$ is at least $(2(\alpha_0+1)+d)$-harmonic.

\smallskip

The preceding considerations and Lemma\,\ref{prep} imply that after applying $\tau$ $(d+1)$-times to $f_{d,\alpha_0+1}$ we are left with terms of Type II.
Hence, the argument of the preceding paragraph yields that $f_{d,\alpha_0+1}(x,y,t)$ is $(2(\alpha_0+1)+d+1)$-harmonic. This establishes the claim.
\end{proof}

We now turn to Theorem\,(\ref{main-3}). Again we start with providing examples.

\begin{example}
\begin{enumerate}
\item The function $f:\Nil\rightarrow\C$ given by $$f_{1,3,7}(x,y,t)=xy^3t^7,$$ is proper $10$-harmonic.
\item The function $f:\Nil\rightarrow\C$ given by $$f_{5,2,4}(x,y,t)=x^5y^2t^4,$$ is proper $8$-harmonic.
\end{enumerate}
\end{example}

Below we prove Theorem\,(\ref{main-3}).

\begin{proof}[Proof of Theorem\,(\ref{main-3}).]
Throughout the proof we make use of the short hand notations
$$r_e(m,n,\alpha)=\lfloor (m+\alpha)/2\rfloor+\lfloor n/2\rfloor+1+\lfloor \alpha/2\rfloor$$
and
$$r_o(m,n,\alpha)=\lfloor (m+\alpha)/2\rfloor+\lfloor (n-1)/2\rfloor+2+\lfloor \alpha/2\rfloor.$$
Thus, the goal is to prove that the function $x^my^nt^{\alpha}$ is $r_e(m,n,\alpha)$-harmonic if $\alpha$ is even and $r_o(m,n,\alpha)$-harmonic if $\alpha$ is odd.

\smallskip

The strategy is to prove the claim by induction on $\alpha$.
In order to settle the induction beginning, we provide (1) for $\alpha=0$ and (2) for $\alpha=1$.\\
By induction one proves easily that
for any $m,n\in\N$, the function $f_{m,n,0}:\R^2\rightarrow\R$ given by
$$f_{m,n,0}(x,y)=x^my^{n}$$ is proper $(\lfloor m/2\rfloor+\lfloor n/2\rfloor+1)$-harmonic.
This settles the claim for $\alpha=0$.\\
From (\ref{tension-Nil}) we get
\begin{equation}\label{a1}
\tau(x^my^{n}\cdot t)=m(m-1)x^{m-2}y^nt+n(n-1)x^my^{n-2}t+2x^{m+1}y^{n-1}.
\end{equation}
Let us first ignore the first two terms on the right hand side of (\ref{a1}).
The third term on the right hand side of (\ref{a1}) is 
$(r_o(m,n,1)-1)$-harmonic. 
Thus, $x^my^{n}\cdot t$ is $r_o(m,n,1)$-harmonic, as claimed.
Hence this establishes the claim for $\alpha=1$, since one shows easily that the first two terms on the right hand side of (\ref{a1}) are at most $(r_o(m,n,1)-1)$-harmonic.

\smallskip

Assume next that we have proven the claim for all $m, n\in\N$ and any $\alpha\leq\alpha_0$
for some $\alpha_0\in\N$. 
In what follows we will assume that $\alpha_0$ is even (since the considerations for $\alpha_0$ odd are completely analogous they are omitted). 
By (\ref{tension-Nil}) we get
\begin{equation}\label{an}
\begin{aligned}
&\tau(x^my^{n}\cdot t^{\alpha_0+1})\\&=m(m-1)x^{m-2}y^n\cdot t^{\alpha_0+1}+n(n-1)x^my^{n-2}\cdot t^{\alpha_0+1}\\
&+2n(\alpha_0+1)x^{m+1}y^{n-1}\cdot t^{\alpha_0}+(\alpha_0+1)\alpha_0(1+x^2)x^my^n\cdot t^{\alpha_0-1}.
\end{aligned}
\end{equation}
By induction assumption the third term on the right hand side of (\ref{an}) is
$r_e(m+1,n-1,\alpha_0)$-harmonic.
Further, the fourth term is $r_o(m+2,n,\alpha_0-1)$-harmonic. Since $\alpha_0$ is even by assumption, their degree of harmonicity agrees. Hence, $x^my^{n}\cdot t^{\alpha_0+1}$ is at least $r_o(m,n,\alpha_0+1)$-harmonic.
It remains to deal with the first and second term of the right hand side of (\ref{an}).
We will restrict ourselves to the first one since the considerations for the second one are analogous.
Applying $\tau$ to $x^{m-2}y^n\cdot t^{\alpha_0+1}$
yields the right hand side of (\ref{an}) with $m$ substituted by $m-2$.
Hence, the third and the fourth term of this right hand side are 
 $r_o(m,n,\alpha_0-1)$-harmonic. 
Hence, $x^my^{n}\cdot t^{\alpha_0+1}$ is again at least $r_o(m,n,\alpha_0+1)$-harmonic.
The remaining two terms can be dealt with inductively. After $\lfloor (m-2)/2\rfloor+\lfloor n/2\rfloor+1$ steps there are only terms of left with exponents of $t$ strictly less than $\alpha_0+1$.
We can thus apply the induction assumption and establish the induction step.
\end{proof}

\begin{remark}
Let $H_p:\R^2\rightarrow\R$ denote a proper $p$-harmonic function.
In view of the two preceding theorems it is natural to ask whether 
functions of the form $$g_{p,\alpha}(x,y,t)=H_p(x,y)\cdot t^{\alpha}$$ are also proper $r$-harmonic for some $r\in\N^+$.
It is easy to prove that they are indeed $r$-harmonic for some $r\in\N^+$ and to give an upper bound on $r$. However, determining the exact value of $r$ is mainly a laborious task which does not give much additional structural insight. 
\end{remark}

\section{The model geometry $\SL2$}\label{section-sl2}

The model space $\SL2$ on Thurston's list is diffeomorphic to the universal cover of the 3-dimensional
Lie group $\SLR 2$ of $2\times 2$ real traceless matrices.  It is well-known that
$\SL2$ can, as a Riemannian manifold, be modeled as $\rn^3$ equipped with the following metric
$$
ds^2=\frac 1{y^2}(dx^2+dy^2)+(dt+\frac{dx}y)^2.
$$
For this fact we refer to \cite{BW-book}.  This metric is different from the
one obtained by lifting the standard metric of $\SLR 2$ to its
universal cover.  It is also clear that it is not a product metric induced by metrics on
$\rn^2$ and $\rn$, respectively.
The Laplace-Beltrami operator on $\SL2$ with the above metric $ds^2$ satisfies
\begin{equation}\label{tension-SL2}
\tau(f)=y^2(\frac {\partial ^2 f}{\partial x^2}+\frac {\partial ^2 f}{\partial y^2})
+2\frac {\partial ^2 f}{\partial t^2}-2y\frac {\partial ^2 f}{\partial x\partial t}.
\end{equation}

In his paper \cite{Gud-14}, Gudmundsson constructed a family of globally defined complex-valued proper biharmonic functions
on the model space $\SL2$.

\begin{example}\label{exam:SL2-1}\cite{Gud-14}
For any non-zero $b\in\cn^6$ the complex-valued function $f_2:\SL2\to\cn$ given by
$$f_2(x,y,t)=b_1xt+b_2t^2+b_3xt^2+b_4yt^2+b_5t^3+b_6yt^3$$
is proper biharmonic on $\SL2$.
\end{example}

From (\ref{tension-SL2}) it is obvious that any polynomial $p_d:\R\rightarrow\C$  in $t$ of degree $d$
is a proper $\lfloor d/2\rfloor +1$-harmonic function on $\SL2$.
This fact and the preceding example can easily be generalized to generate $r$-harmonic maps for any $r\in\N^+$.

\begin{lemma}
Let $p_d:\R\rightarrow\C$ be a polynomial of degree $d$.
Then \begin{enumerate}
\item $f:\SL2\rightarrow\C$ given by
 $$f(x,y,t)=p_d(t)\cdot y$$ is a proper $r$-harmonic function on $\SL2$, where $r=\lfloor d/2\rfloor +1$;
 \item
 $g:\SL2\rightarrow\C$ given by
 $$g(x,y,t)=p_d(t)\cdot x$$ is a proper $r$-harmonic function on $\SL2$, where $r=\lceil d/2\rceil +1$.
 \end{enumerate}
\end{lemma}

\begin{proof}
Since $f$ is linear in $y$ and does not depend on the variable $x$ we get
$$\tau(f)=2\cdot\frac {\partial ^2 f}{\partial t^2}.$$
This implies $(1)$.

\smallskip

To prove $(2)$ we proceed by induction. Clearly, $P_0(t)x=cx$ is proper harmonic.
Assume that $(2)$ holds true for any integer $d$ less than some given integer $d_0$.
From (\ref{tension-SL2}) we obtain
\begin{equation}
\tau (g)=2x\frac {\partial ^2 P_{d_0}}{\partial t^2}-2y\frac {\partial P_{d_0}}{\partial t}.
\end{equation}
The first summand is of the form $P_{d_0-2}(t)x$. 
By induction, this implies that this summand is proper $\lceil d_0/2\rceil$-harmonic.
The second summand is of the form considered in $(1)$ and is thus a proper $\lfloor (d_0-1)/2\rfloor +1$-harmonic function. Note $\lfloor (d_0-1)/2\rfloor +1=\lceil d_0/2\rceil$.
When applying the operator $\tau$ $n$-times to $g$, the different resulting summands can not cancel out since there needs to be exactly one summand linear in $x$ (as long as $n<r$ of course).
Thus, since $\tau$ is linear, this establishes the claim.
\end{proof}

\section{The product spaces $\H^2\times\rn$ and $\s^2\times\rn$}\label{section-product}
In this section we provide proper $r$-harmonic maps on the product spaces $\H^2\times\rn$ and $\s^2\times\rn$ on Thurston's list.

\subsection{The product space $\H^2\times\rn$}
Let $\H^2$ be the hyperbolic disc endowed with its standard Riemannian metric of constant curvature $-1$.  Further, we equip the product space $\H^2\times\rn$ with its product metric.  It is straightforward to prove, see \cite{Gud-14}, that in the standard global coordinates $(z,t)$ on $\H^2\times\rn$, the operators $\tau$ and $\kappa$ are then given by
\begin{equation*}\label{tension-H2-R}
\tau(f)=4(1-z\bar z)^2\frac {\partial ^2 f}{\partial z\partial \bar z}+\frac {\partial ^2 f}{\partial t^2}
\end{equation*}
and
\begin{equation*}\label{conformality-H2-R}
\kappa(f,h)=2(1-z\bar z)^2(
\frac{\partial f}{\partial z}\frac{\partial h}{\partial \bar z}
+\frac{\partial f}{\partial \bar z}\frac{\partial h}{\partial z})
+\frac{\partial f}{\partial t}\frac{\partial h}{\partial t}.
\end{equation*}

In the following lemma we generalise Corollary 7.1 of \cite{Gud-14} and provide proper $r$-harmonic functions on the product space $\H^2\times\rn$.

\begin{lemma}\label{h2}
Let the functions $f,g:\H^2\to\cn$ be holomorphic and $P:\rn\to\cn$ be a complex-valued polynomial
$$P(t)=\sum_{k=0}^{2r-1}b_kt^k,$$ such that $(b_{2r-2},b_{2r-1})\neq 0$.  Then the function
$$F(z,t)=(f(z)+g(\bar z))\cdot P(t)$$
is proper $r$-harmonic on the product space $\H^2\times\rn$.
\end{lemma}

\begin{proof}
Since the functions $f$ and $g$ are holomorphic, we have
$$\tau(F)(z,t)=(f(z)+g(\bar z))\cdot \frac {\partial ^2 P}{\partial t^2}.$$
Thus the claim is established.
\end{proof}

\subsection{The product space $\s^2\times\rn$}
The construction of $r$-harmonic functions on the product space $\s^2\times\rn$ works as the one for $\H^2\times\rn$.  
Due the maximum principle for harmonic functions we need to consider the punctured sphere $\Sigma^2=\s^2\setminus\{p\}$ instead of $\s^2$.  Following \cite{Gud-14},  we model $\Sigma^2$ as the complex plane equipped with its historic Riemannian metric
$$ds^2=\frac 4{(1+(x^2+y^2))^2}(dx^2+dy^2)$$
of constant curvature +1. 

\smallskip

By the following result we provide local proper $r$-harmonic maps on $\s^2\times\rn$.
Since the proof is as elementary as that of Lemma \ref{h2} we have chosen to omit it here.

\begin{lemma}
Let the functions $f,g:\Sigma^2\to\cn$ be holomorphic on the punctured sphere.
Further let $P:\rn\to\cn$ be a complex-valued polynomial
$$P(t)=\sum_{k=0}^{2r-1}b_kt^k,$$ such that $(b_{2r-2},b_{2r-1})\neq 0$.
Then the function
$$
F(z,\bar z,t)=(f(z)+g(\bar z))\cdot P(t)
$$ is proper $r$-harmonic on the product space $\Sigma^2\times\rn$.
\end{lemma}

\end{document}